\documentclass[reqno]{amsart}
\usepackage{amsmath,amssymb}
\usepackage{relsize}
\usepackage{graphicx,color}
\usepackage[all]{xy}
\theoremstyle{plain}
\newtheorem{introtheorem}{Theorem}

\newtheorem{theorem}{Theorem}[section]
\newtheorem{proposition}[theorem]{Proposition}
\newtheorem{lemma}[theorem]{Lemma}
\newtheorem{corollary}[theorem]{Corollary}

\theoremstyle{definition}
\newtheorem{definition}[theorem]{Definition}
\newtheorem{example}[theorem]{Example}

\theoremstyle{remark}
\newtheorem{remark}[theorem]{Remark}

\def\E{{\mathcal E}}

\def\D{{\mathcal D}}

\def\A{{\mathcal A}}
\def\G{{\mathcal G}}

\def\L{{\mathbb L}}
\def\L{{\Bbb  L}}

\def\G{\mathcal G}

\def\cat0{\mathrm{cat}_0}

\def\dim{\mathrm{dim}}

\def\aut{\mathrm{aut}}

\begin{document}

\title[ ]
{On the  GROUP OF SELF-HOMOTOPY EQUIVALENCES OF AN ELLIPTIC SPACE }

\author{Mahmoud Benkhalifa}
\address{Department of Mathematics. Faculty of  Sciences, University of Sharjah. Sharjah, United Arab Emirates}

\email{mbenkhelifa@sharjah.ac.ae}

%\date{\today}


\begin{abstract} 
	Let $X$ be  a simply connected rational  elliptic  space of formal dimension $n$ and let $\E(X)$ denote the group of homotopy classes of self-equivalences of $X$. If $X^{[k]}$ denotes  the $k^{\text{th}}$ Postikov section  of $X$ and $X^{k}$ denotes  its  $k^{\text{th}}$ skeleton, then making use of the models of Sullivan and Quillen we  prove that	$\E(X)\cong\E(X^{[n]})$ and   if $n>m=max\big\{k \,| \,\pi_{k}(X)\neq 0\big\}$ and $\E(X)$ is finite, then $\E(X)\cong\E(X^{m+1})$. Moreover, in case when $X$ is 2-connected,  we show that if  $\pi_{n}(X)\neq0$,  then the group $\E(X)$ is infinite.
\end{abstract}
\maketitle
\section{Introduction}
A simply connected rational topological space $X$ is  elliptic if  both  
	$ H^*(X,\Bbb Q)$ and $\pi_{*}(X)$ are finite  dimensional. Let us call $n = \mathrm{max}\{i:\, H^i(X,\Bbb Q)\neq 0\}$ the formal dimension of $X$. It is well known (\cite{Au}, Theorem 4.2) that a such space satisfies $\pi_{i}(X)=0,$ for $i\geq 2n.$
	
Therefore, if $X^{[k]}$ denotes  the $k^{\text{th}}$ Postikov section  of $X$ and $X^{k}$ denotes  its  $k^{\text{th}}$ skeleton,    on the one hand  $X$ coincides with  $X^{[2n-1]}$ and on the other hand, as its   formal dimension is  $n$, the space  $X$  coincides with  $X^n$.

This paper is directed towards an understanding of the group of homotopy classes of self-equivalences $\E(X)$, where $X$ is a simply connected rational  elliptic  space of formal dimension $n$.  As is well known, the homotopy theory of rational spaces is equivalent,  by  Sullivan's work,  to the homotopy theory of minimal, differential, graded commutative $\Bbb Q$-algebras and by Quillen's work to the homotopy theory of  differential, graded Lie $\Bbb Q$-algebras.  Those  algebras provide an effective algebraic setting to work in, so working algebraically we establish  the following theorem which will be split in Theorem \ref{t0}  and Theorem \ref{t09} later on.
\begin{introtheorem}
Let $X$ be  a simply connected rational  elliptic  space of formal dimension $n$. Then
$$\E(X)= \E(X^{[2n-1]})\cong \E(X^{[2n-2]})\cong\dots\cong\E(X^{[n]})$$
 Moreover if $n>m=max\big\{k \,| \,\pi_{k}(X)\neq 0\big\}$ and  $\E(X)$ is finite, then 
$$\E(X)\cong\dots\cong\E(X^{m+2})\cong\E(X^{m+1}).$$
\end{introtheorem}
\medskip
The idea of  getting some information regarding the (in)finiteness of the groups $\E(X)$  within the framework of Sullivan model   traces back to the results of Arkowitz and Lupton \cite{AL}  in which they exhibited conditions under which $\E(X)$  is finite or infinite,  where $X$ is a rational space having a 2-stage Postnikov-like decomposition (for example, rationalizations of homogeneous spaces). In the same spirit we establish the following result  which will be Theorem \ref{t06} later on.
\begin{introtheorem}
	Let $X$ be  a $\rm{2}$-connected rational  elliptic  space of formal dimension $n$.  If   $\pi_{n}(X)\neq0$,  then the group $\E(X)$ is infinite.
\end{introtheorem}
\medskip

In \cite{CV},  Costoya and Viruel  proved the remarkable
result that every finite group $G$ occurs as $G= \E(X)$ for some  elliptic rational space  $X$ having  formal dimension $n=208+80|V|$,  where  $V$ is a certain finite graph associated with $G$ and $|V|$ denotes  the order of $V$. The space $X$ is   constructed such that  $\pi_{k}(X)= 0$ for all $k\geq 120$. As a consequence of the main theorem of this paper we show that   every finite group $G$ can be realised by a rational space $X$ whose formal dimension does not depend on the order of $G$. Precisely,  we can ameliorate Costoya and Viruel  theorem by showing that $G= \E(X)$ for some   rational space  $X$ having  formal dimension $n=120$. 

\medskip

For a  space $X$, let  $[X,X]$ denote the monoid of homotopy classes of  self-maps of the  space $X$ and let
$$\A_{\#}^{k}(X)=\Big\{f\in[X,X]\,\mid\,\pi_{i}(f):\pi_{i}(X)\overset{\cong}{\longrightarrow} \pi_{i}(X) \text{ for any } i\leq k\}$$ 
In \cite{CL}, Choi and Lee  introduced the concept of the  self-closeness number  defined as follows
	$$N\E(X)=min\Big\{k\,\mid\,\A_{\#}^{k}(X)=\E(X)\Big\}$$
Notice that the study of this numerical  homotopy invariant by means of  Sullivan model in algebraic setting  is an  interesting  problem (see for instance \cite{OY}).  From the main result in this paper we establish the following result  which will be Theorem \ref{t} later on.
\begin{introtheorem}
	If $X$  is a simply connected rational  elliptic  space of formal dimension $n$, then 
	$$N\E(X)\leq n$$
\end{introtheorem}
\begin{remark}
Theorem 3 is a weaker version of (\cite{CL}, Theorem 2) in which the same result is proved,  using topological arguments,  for a  CW-complex  $X$ of dimension $n$. Our proof is based on the analysis of the Sullivan model of $X$.
\end{remark}

The paper is organized as follows.
In section 2, we recall the basic properties of the Quillen model and Sullivan model in rational homotopy theory, the Whitehead  exact sequences  as well as the important properties of elliptic spaces. Then   we formulate and prove the main theorems in the algebraic setting. In section 3, a mere transcription of the above results in the topological context is given and some examples are provided  illustrating our results.
\section{main results}
\subsection{Quillen model and Sullivan model in rational homotopy theory}
We briefly recall  Quillen's   differential graded Lie algebra and Sullivan's commutative differential algebra  frameworks  for rational homotopy theory. All the materials can be founded \cite{Au,FHT}.

 If $X$ is a simply connected rational  CW complex of finite type,  then there exists a free commutative cochain algebra $(\Lambda
V,\partial)$  called  the  Sullivan model of $X$, unique up to isomorphism, which determines completely the homotopy type of the space $X$. Moreover Sullivan  model recovers homotopy data via the identifications$\colon$
$$\mathrm{Hom}\big(\pi_{*}(X),\Bbb Q\big)\cong V^{*} \ \  \hbox{,} \ \  H^*(X; \Bbb Q) \cong H^{*}(\Lambda
V,\partial)\ \  \hbox{and} \ \  \E(X) \cong \aut(\Lambda
	V,\partial)/\simeq, $$
where $\aut(\Lambda
	V,\partial)/\simeq$  is the group  of  homotopy  cochain self-equivalences  of  $(\Lambda
V,\partial)$ modulo the relation of homotopy between free commutative cochain algebras (see \cite{FHT}).   We write $$\E( \Lambda V) =   \aut(\Lambda V, \partial)/\simeq$$
for this group.

\bigskip

Dually, if $X$ is a simply connected rational  CW complex of finite type,  then there exists a differential graded Lie algebra $(\L(W),\delta)$ called the Quillen model of $X$,  unique up to isomorphism, which determines completely the homotopy type of the space $X$.   The Quillen  model recovers homotopy data via the identifications$\colon$
$$\pi_*(X)  \cong H_{*-1}(\L(W))  \ \  \hbox{and} \ \  H_*(X,\Bbb Q) \cong W_{*}.$$
Quillen's theory directly implies an identification    $$\E(X) \cong \aut(\L(W))/\simeq, $$    where the latter  is the group    of  homotopy differential graded  Lie self-equivalences of  $(\L(W), \delta)$ modulo the relation of  homotopy between differential graded  Lie automorphisms   (see \cite[pp.425-6]{FHT}).   We write $$\E(\L(W)) =   \aut(\L(W), \partial)/\simeq$$
for this group.
\begin{definition}[\cite{Benk3}, Definition 2.6]\label{d3}
	Given a simply connected free commutative cochain algebra  $(\Lambda (V^q\oplus V^{\leq n}),\partial)$, where  $q>n$ and   let   $b^{q}:V^q\to H^{q+1}(\Lambda(V^{\leq n}))$ be the linear map  defined as follows
	\begin{equation}
	b^{q}(v)=[\partial(v)]\,\,\,\,\,\,,\,\,\,\,\, v\in V^{q}.\label{37}
	\end{equation}
	Here $[\partial(v)]$  denotes
	the cohomology class of $\partial(v)\in
	(\Lambda V^{\leq n})^{q+1}$.

\medskip 

	We define  $\mathcal{D}^{q}_{n}$ to be the subgroup   of $\aut(V^{q})\times \mathcal{E}(\Lambda V^{\leq
		n})$ consisting of the pairs $(\xi,[\alpha])$  making  the following
	diagram commutes
\begin{equation}\label{13}
\begin{picture}(300,90)(0,30)
\put(60,100){$V^{q}\hspace{1mm}\vector(1,0){153}\hspace{1mm}V^{q}$}
\put(69,76){\scriptsize $b^{q}$} \put(238,76){\scriptsize $b^{q}$}
\put(66,96){$\vector(0,-1){38}$} \put(235,96){$\vector(0,-1){38}$}
\put(155,103){\scriptsize $\xi$} \put(145,52){\scriptsize
	$H^{q+1}(\alpha)$} \put(35,46){$H^{q+1}(\Lambda V^{\leq n})
	\hspace{1mm}\vector(1,0){110}\hspace{1mm}H^{q+1}(\Lambda V^{\leq n})
	\hspace{1mm}$}
\end{picture}
\end{equation}
\end{definition}
\begin{theorem}[\cite{Benk3}, Theorem 1.1]\label{t8}
	There exists a 
  short exact sequence of groups
	\begin{equation}\label{34}
	\mathrm{Hom}\big(V^q, H^{q}(\Lambda(V^{\leq n}))\big) \rightarrowtail
	\mathcal{E}(\Lambda(V^{q}\oplus V^{\leq n}))\overset{\Psi}{
		\twoheadrightarrow}\D^{q}_{n}
	\end{equation}
	where $\Psi([\alpha])=(\widetilde{\alpha}^{q},[\alpha_{n}])$.  Here $\widetilde{\alpha}^{q}:V^q\to V^q$ is the isomorphism induced by $\alpha$ on the indecomposables and $\alpha_{n}$ is the restriction of $\alpha$ to $\Lambda V^{\leq n}$.
\end{theorem}
\begin{corollary}\label{c1}
Assume that the linear map  $b^{q}$ is an isomorphism, then 
	$$\D^{q}_{n}\cong  \mathcal{E}(\Lambda V^{\leq
		n})$$
\end{corollary}
\begin{proof}
	As  $b^{q}$ is an isomorphism, then from the commutative diagram (\ref{13}) we deduce that $\xi=(b^q)^{-1}\circ H^{q+1}(\alpha)\circ b^q$. Therefore the map
	$$ \mathcal{E}(\Lambda V^{\leq
		n})\to \D^{q}_{n} \,\,\,\,\,\,\,\,\,\,\,\,\,\,\,\,\,,\,\,\,\,\,\,\,\,\,\,\,\,\,\,\,\,[\alpha]\longmapsto\Big((b^q)^{-1}\circ H^{q+1}(\alpha)\circ b^q,[\alpha]\Big)$$
	is an isomorphism. 
\end{proof}

\begin{definition}[\cite{BS}, Definition 2.1]\label{d5}
	Given a simply connected free differential graded Lie   $(\L(W_q\oplus W_{\leq k}),\delta)$ where  $q>k$ and   let   $b_{q}:W_q\to H_{q-1}(\L (W_{\leq k}))$ be the linear map  defined as follows
	\begin{equation}
	b_{q}(v)=[\delta(v)]\,\,\,\,\,\,,\,\,\,\,\, v\in W_{q}\label{037}
	\end{equation}
	Here $[\delta(v)]$  denotes
	the homology class of $\delta(v)\in
	\L_{q-1} (W_{\leq k})$.
	
	\medskip

		We define  $\mathcal{R}^{q}_{k}$ to be subgroup   of $\aut(W_{q})\times \mathcal{E}(\L (W_{\leq
			k})$ consisting of the pairs $(\xi,[\alpha])$  making  the following
		diagram commutes
	\begin{equation}		\label{020}
		\begin{picture}(300,90)(0,30)
		\put(60,100){$W_{q}\hspace{1mm}\vector(1,0){150}\hspace{1mm}W_{q}$}
		\put(69,76){\scriptsize $b_{q}$} \put(238,76){\scriptsize $b_{q}$}
		\put(66,96){$\vector(0,-1){38}$} \put(235,96){$\vector(0,-1){38}$}
		\put(155,103){\scriptsize $\xi$} \put(145,52){\scriptsize
			$H_{q-1}(\alpha)$} \put(35,46){$H_{q-1}(\L (W_{\leq k}))
			\hspace{1mm}\vector(1,0){110}\hspace{1mm}H_{q-1}(\L (W_{\leq k}))
			\hspace{1mm}$}
		\end{picture}
			\end{equation}
\end{definition}
\begin{theorem}[\cite{BS}, Theorem 2.6]\label{t08}
	There exists 
a   short exact sequence of groups
	\begin{equation}\label{034}
	\mathrm{Hom}\big(W_q, H_{q}(\L (W_{\leq
		k}))\big) \rightarrowtail
	\mathcal{E}(\L(W_{q}\oplus W_{\leq k}))\overset{\lambda}{
		\twoheadrightarrow}\mathcal{R}^{q}_{k}
	\end{equation}
		where $\lambda([\alpha])=(\widetilde{\alpha}_{q},[\alpha_{k}])$.  Here $\widetilde{\alpha}_{q}:W_q\to W_q$ is the isomorphism induced by $\alpha$ on the indecomposables and $\alpha_{k}$ is the restriction of $\alpha$ to $\L (W_{\leq k})$
\end{theorem}
\begin{corollary}\label{c0}
	Assume that the linear map  $b_{q}$ is an isomorphism, then 
	$$\mathcal{R}^{q}_{k}\cong  \mathcal{E}(\L (W_{\leq
		k}))$$
\end{corollary}
\begin{proof}
	As  $b_{q}$ is an isomorphism, then from the commutative diagram (\ref{020}) we deduce that $\xi=(b_q)^{-1}\circ H_{q-1}(\alpha)\circ b_q$. Therefore the map
	$$ \mathcal{E}(\L (W_{\leq
		k}))\to\mathcal{R}^{q}_{k} \,\,\,\,\,\,\,\,\,\,\,\,\,\,\,\,\,,\,\,\,\,\,\,\,\,\,\,\,\,\,\,\,\,[\alpha]\longmapsto\Big((b_q)^{-1}\circ H_{q-1}(\alpha)\circ b_q,[\alpha]\Big)$$
	is an isomorphism.  
\end{proof}
\subsection{Whitehead exact sequences in rational homotopy theory}
To  every  free differential graded Lie algebra 	$(\L(W),\delta)$ such that $W_{1}=0$, we can assign (see  \cite{Benk10,Benk0,BS,Benk14}) the following long  exact   sequence
\begin{equation}\label{23}
\cdots \rightarrow W_{n+1}\overset{b_{n+1}}{\longrightarrow }%
H_{n}(\L (W_{\leq n-1}))\rightarrow H_{n}(\L(W))\rightarrow
W_{n}
\overset{b_{n+1}}{\longrightarrow }\cdots
\end{equation}
called the Whitehead exact sequence of $(\L(W),\delta)$, where  $b_{*}$ is the graded  linear map defined in (\ref{037}).  Hence if  $X$ is a 2-connected  rational space of finite type and if $(\L (W),\partial)$ is its Quillen's  model, then   the properties of  this model imply
$$\pi_{n}(X)\cong H_{n-1}(\L (W))\,\,\,\,\,,\,\,\,H_{n}(X,\Bbb Q)\cong W_{n-1}
\,\,,\,\,\,\pi_{n}(X^{n-1})\cong H_{n-1}(\L (W_{\leq n-2}))$$
Here $X^{n}$ for the $n^{\text{th}}$ skeleton   of $X$.  
Therefore the Whitehead exact sequence of this model can be written as 
	\begin{equation}\label{25}
\cdots \rightarrow
H_{n+1}(X)\overset{}{\rightarrow }%
\pi_{n}(X^{n-1})\overset{}{\rightarrow} \pi
_{n}(X)\overset{}{\rightarrow }%
H_{n}(X)\overset{}{\rightarrow }\cdots
\end{equation}

Likewise, let $(\Lambda V,\partial)$ be a  simply connected free commutative cochain algebra. In \cite{Benk1,Benk2,Benk4,Benk15,Benk16},  it is shown that with $(\Lambda V,\partial)$  we can associate the following  long exact sequence
\begin{equation}\label{17}
\cdots\to V^{k}\overset{b^{k}}{\longrightarrow} H^{k+1}(\Lambda V^{\leq k-1})\to H^{k+1}(\Lambda V)\to V^{k+1}\overset{b^{k+1}}{\longrightarrow}\cdots
\end{equation}
called the Whitehead exact sequence of $(\Lambda V,\partial)$, where  $b^{*}$ is the graded  linear map defined in (\ref{37}). Thus,  if  $X$ is  a simply connected rational
space of finite type and $(\Lambda(V),\partial)$ is its   Sullivan minimal
model, then by virtue of  the properties of  this model we obtain the following identifications
$$H^{k}(X,\Bbb Q)\cong H^{n}(\Lambda V
)\,\,\,,\,\,\,H^{k+1}(X^{[k]},\Bbb Q)\cong H^{k+1}(\Lambda V^{\leq k})\,\,\,,\,\,\,V^{k}\cong\mathrm{Hom}(\pi_{k}( X),\Bbb Q)$$
Here $X^{[k]}$ for the $k^{\text{th}}$ Postikov section  of $X$.  Therefore the Whitehead exact sequence of this model can be written as
	\begin{equation}\label{21}
\cdots \rightarrow \mathrm{Hom}(\pi_{k}( X),\Bbb Q)\overset{}{\rightarrow }%
H^{k+1}(X^{[k]})\rightarrow H^{k+1}(X)\rightarrow
\mathrm{Hom}(\pi_{k+1}( X),\Bbb Q)%
\overset{}{\rightarrow }\cdots
\end{equation}
\begin{theorem}
	\label{00}
	If $X$  is  a $\rm{2}$-connected 
	rational space  of finite type, then
	\begin{equation}\label{24}
	H^{k+1}(X^{[k]},\Bbb Q)\cong \mathrm{Hom}(\pi_{k}(X^{k-1}),\Bbb Q)\,\,\,\,\,\,,\,\,\,\,\,\,k\geq 2
	\end{equation}
\end{theorem}
\begin{proof} 
	Applying  the exact functor $\mathrm{Hom}(., \Bbb Q)$ to the	exact sequence (\ref{25})  we obtain
	\begin{equation}\label{22}
	\cdots \leftarrow
	H^{n+1}(X,\Bbb Q)\overset{}{\leftarrow }%
	\mathrm{Hom}(\pi_{n}(X^{n-1}), \Bbb Q)	\overset{}{\leftarrow}\mathrm{Hom}(\pi
	_{n}(X), \Bbb Q) \overset{}{\leftarrow }%
	H^{n}(X,\Bbb Q)	\overset{}{\leftarrow }\cdots
	\end{equation}
	Taking into  account that
	\begin{itemize}
		\item All  groups involved   are vector spaces of finite dimensions 
		\item The two maps $H^{n}(X,\Bbb Q)\to \mathrm{Hom}(\pi
		_{n}(X), \Bbb Q) $ appearing  in   (\ref{21}) and (\ref{22}) are the same morphism
		\item $	\mathrm{Hom}(H_{*}(X,\Bbb Q), \Bbb Q)=H^{*}(X,\Bbb Q)$
	\end{itemize}
	and  by comparing  the  sequences (\ref{21}), (\ref{22}) we get  (\ref{24}).
\end{proof}
\subsection{Elliptic  algebras}
Recall that (see \cite{Au, FHT}) 
	a simply connected free differential graded commutative   algebra $(\Lambda V,\partial)$  is called elliptic if both 
	$H^*(\Lambda V)$ and $V^*$ are finite dimensional. Let us call $n = \mathrm{max}\{i:\, H^i(\Lambda V)\neq 0\}$ the formal dimension of $(\Lambda V,\partial)$. The following theorem mentions some important properties of   elliptic algebras.
\begin{theorem}\label{t1}(\cite{Au},  $\rm{Theorem\, 7.4.2}$).
	Suppose $(\Lambda V,\partial)$  is  simply connected and  elliptic of formal dimension $n$. Then
	\begin{enumerate}
		\item $\dim\, V^{odd}\geq \dim\, V^{even}$.
		\item If $\{x_j\}$ is a basis of $ V^{odd}$ and $\{y_j\}$ is a basis of $ V^{even}$, then
		$$n=\sum_{}\vert x_{j}\vert- \sum_{}(\vert y_{j}\vert-1).$$
		\item $ \sum_{}\vert y_{j}\vert\leq n$ and  $ \sum_{}\vert x_{j}\vert\leq 2n-1.$ 
		\item  $V^{i}=0,$ for $i\geq 2n.$
	\end{enumerate}
\end{theorem}
From the property (3) we can derive
\begin{corollary}\label{c3}
	Suppose $(\Lambda V,\partial)$  is simply connected and elliptic of formal dimension $n$. Then 
 $V^{i}=0,$ for $i>n$  and $i$ even.
\end{corollary}
\begin{lemma}\label{l02}
		Suppose $(\Lambda V,\partial)$  is  simply connected and  elliptic of formal dimension $n$. For every $k$ such that $2k>n$ we have $$H^{2k+1}(\Lambda V^{\leq
		2k-1})=0$$
\end{lemma}
\begin{proof}
	From the long exact sequence of cohomology associated to the inclusion $(\Lambda V^{\leq
		2k-1} ,\partial)\subseteq (\Lambda V,\partial)$ we get
	$$\dots\to H^{2k}\Big(\Lambda V/\Lambda V^{\leq
		2k-1}\Big)\to H^{2k+1}(\Lambda V^{\leq
		2k-1})\to H^{2k+1}(\Lambda V)\to\cdots$$
	But $H^{2k}\Big(\Lambda V/\Lambda V^{\leq
		2k-1}\Big)\cong V^{2k}$ and by Corollary \ref{c3} we have $V^{2k}=0$ for $2k>n$. Therefore the map $H^{2k+1}(\Lambda V^{\leq
		2k-1})\to H^{2k+1}(\Lambda V)$ is injective. As $(\Lambda V,\partial)$ has  formal dimension $n$, it follows that $H^{2k+1}(\Lambda V)=0$ for $2k\geq n$,  implying that $H^{2k+1}(\Lambda V^{\leq
		2k-1})=0.$
\end{proof}
Using the exact sequence (\ref{17}), Corollary   \ref{c3} and Lemma \ref{l02}, we derive the following result
\begin{corollary}\label{c4}
	Suppose $(\Lambda V,\partial)$  is   simply connected and   elliptic of formal dimension $n$. Then the  linear map $b^{i}:V^{i}\overset{}{\longrightarrow} H^{i+1}(\Lambda V^{\leq i-1})$ is
\begin{itemize}
	\item 		An isomorphism for $i>n$  and $i$ odd.
		\item 	Nil for for $i>n$  and $i$ even.
\end{itemize}
\end{corollary}
\begin{theorem}\label{t2}
	Suppose $(\Lambda V,\partial)$  is  simply connected and   elliptic of formal dimension $n$.  Then
	$$\E(\Lambda V)\cong \E(\Lambda V^{\leq n})$$
\end{theorem}
\begin{proof}
First   note that since $(\Lambda V,\partial)$  is   elliptic of formal dimension $n$, by  the property (3) of Theorem \ref{t1}, we  derive that $\E(\Lambda V)=\E(\Lambda V^{\leq 2n-1}).$ Next using 	Theorem \ref{t8} we obtain the following short exact sequence
\begin{equation}\label{01}
\mathrm{Hom}\big(V^{2n-1}, H^{2n-1}(\Lambda V^{\leq 2n-2})\big) \rightarrowtail
\mathcal{E}(\Lambda  V^{\leq 2n-1})\overset{\mathrm{}}{
\twoheadrightarrow}\D_{2n-2}^{2n-1}
\end{equation}
But  $V^{ 2n-2}=0$ and   $b^{2n-1}$ is an isomorphism, so using Corollary   \ref{c1} and Lemma  \ref{l02}, the sequence (\ref{01})  implies
$$\mathcal{E}(\Lambda( V^{\leq 2n-1}))
	\cong\mathcal{E}(\Lambda( V^{\leq 2n-3})$$
	Again using 	theorem \ref{t8} we obtain the following short exact sequence
	\begin{equation*}\label{02}
	\mathrm{Hom}\big(V^{2n-3}, H^{2n-3}(\Lambda V^{\leq 2n-4})\big) \rightarrowtail
	\mathcal{E}(\Lambda  V^{\leq 2n-3})\overset{\mathrm{}}{
		\twoheadrightarrow}\D_{2n-2}^{2n-3}
	\end{equation*}
but  $V^{ 2n-4}=0$ and   $b^{2n-3}$ is an isomorphism, so using Corollary   \ref{c1} and Lemma  \ref{l02}, the sequence (\ref{01})  implies
	$$\mathcal{E}(\Lambda( V^{\leq 2n-3}))
	\cong\mathcal{E}(\Lambda( V^{\leq 2n-5})$$
	Continuing this process by  using the same arguments,  we end up with the following formula
$$\mathcal{E}(\Lambda( V^{\leq 2n-3}))
\cong\mathcal{E}(\Lambda( V^{\leq 2n-5})\cong\dots\cong\mathcal{E}(\Lambda( V^{\leq n})\,\,\,,\,\,\,\, \text{ if $n$ is odd}$$
and 
$$\mathcal{E}(\Lambda( V^{\leq 2n-4}))
\cong\mathcal{E}(\Lambda( V^{\leq 2n-6})\cong\dots\cong\mathcal{E}(\Lambda( V^{\leq n+1})\,\,\,,\,\,\,\, \text{ if $n$ is even.}$$
In the case when $n$ is even,  according to  Corollary   \ref{c1} and Lemma  \ref{l02}, the sequence (\ref{01})  implies
$$\mathrm{Hom}\big(V^{n+1}, H^{n+1}(\Lambda V^{\leq n})\big) \rightarrowtail
\mathcal{E}(\Lambda V^{\leq n+1})\overset{\mathrm{}}{
	\twoheadrightarrow}\E(\Lambda V^{\leq n})$$
Finally, Lemma   \ref{l02}  assures   that  $H^{n+1}(\Lambda V^{\leq n})=0$. Hence $\mathcal{E}(\Lambda V^{\leq n+1})\cong\E(\Lambda V^{\leq n})$
\end{proof}
\begin{proposition}\label{p1}
	Let  $(\Lambda V,\partial)$  be a  simply connected free differential graded  algebra.  If the group  $\E(\Lambda V^{\leq n})$ is finite, then the linear map $b^{n}$ is injective. 
\end{proposition}
\begin{proof}
	First  Theorem \ref{t8} implies that 
	\begin{equation}\label{06}
	\mathrm{Hom}\big(V^{n}, H^{n}(\Lambda V^{\leq n-1})\big) \rightarrowtail
	\mathcal{E}(\Lambda V^{\leq n})\overset{}{
		\twoheadrightarrow}\D^{n}_{n-1}
	\end{equation}
	Assume that  $b^{n}$ is not  injective and let $v\neq 0\in V^{n}$ such that $b^{n}(v)=0$. Choose $\{v,v_{1},\dots,v_{k}\}$ as a basis of $V^{n}$ and define
	 $$\xi^{a}(v)=av\,\,\,\,\,\,\,\,,\,\,\,\,\,\,\,a\neq 0\in \Bbb Q\,\,\,\,\,\,\,\,\,\,\,\,\,\,\,\,\,\,\,\,,\,\,\,\,\,\,\,\,\,\,\,\,\,\,\,\,\,\,\xi^{a}(v_{i})=v_{i}$$
Clearly the pair $(\xi^{a},[id])\in\aut(V^{n})\times \mathcal{E}(\Lambda V^{\leq
	n-1})$ for every $a\neq 0\in \Bbb Q$ and   makes following
diagram commute

\begin{picture}(300,90)(0,30)
\put(60,100){$V^{n}\hspace{1mm}\vector(1,0){153}\hspace{1mm}V^{n}$}
\put(69,76){\scriptsize $b^{n}$} \put(238,76){\scriptsize $b^{n}$}
\put(66,96){$\vector(0,-1){38}$} \put(235,96){$\vector(0,-1){38}$}
\put(155,103){\scriptsize $\xi^a$} \put(145,50){\scriptsize
	$id$} \put(35,46){$H^{n+1}(\Lambda V^{\leq n-1})
	\hspace{1mm}\vector(1,0){110}\hspace{1mm}H^{n+1}(\Lambda V^{\leq n-1})
	\hspace{1mm}$}
\end{picture}

\noindent Therefore   $(\xi^{a},[id])\in\D^{n}_{n-1}$ for every $a\neq 0\in \Bbb Q$ implying that the group $\D^{n}_{n-1}$ is infinite. Consequently,  the group  $\E(\Lambda V^{\leq n})$ is also infinite according to the exact sequence  (\ref{06}) 
\end{proof}

\section{Topological applications}
\begin{definition}
	A  simply connected rational space $X$ is called  elliptic if its Sullivan model is elliptic. 
\end{definition}
In this case  the formal dimension of $X$ is defined as the formal dimension  of its Sullivan model. 
\begin{remark}
\label{r1}	
By virtue of Theorem \ref{t1} we conclude that if  $X$  is a simply connected rational  elliptic space of formal dimension $n$, then its  Quillen model can be written as $(\L (W_{\leq
	n-1}),\delta)$  and   its   Sullivan model as $(\Lambda V^{\leq 2n-1},\partial)$.
\end{remark}
A mere transcription of the Theorem \ref{t2} in the topological context, using the properties of the Sullivan model,  implies the following  theorem.
\begin{theorem}\label{t0}
	Let $X$ be  a rational  elliptic  space of formal dimension $n$. Then
	$$\E(X)= \E(X^{[2n-1]})\cong \E(X^{[2n-2]})\cong\dots\cong\E(X^{[n]})$$
	Here $X^{[k]}$ denotes  the $k^{\text{th}}$ Postnikov section  of $X$. 
\end{theorem}

Combining the model of Quillen and the model of Sullivan we obtain
\begin{theorem}\label{t09}
Let  $X$ be  a simply connected rational  elliptic  space of formal dimension $n$ such  that $n>m=max\big\{k \,\,\,| \,\,\,\pi_{k}(X)\neq 0\big\}$ and $\E(X)$ is finite.  Then \begin{equation}\label{016}
\E(X)\cong\dots\cong\E(X^{m+2})\cong\E(X^{m+1})
\end{equation}
\end{theorem}
\begin{proof}
By hypothesis the  Quillen model of $X$ has the form $(\L (W_{\leq n-1}),\delta)$, where $W_{n-1}\neq 0$ and its is  Sullivan model  has the form $(\Lambda V^{\leq m},\partial)$, where  $V^{m}\neq 0$. Recall that 
\begin{equation}\label{014}
V^{*}\cong \mathrm{Hom}\Big(H_{*-1}(\L(W_{\leq n-1})),\Bbb Q\Big)
\end{equation}
Since $n>m$ and by (\ref{014}), we deduce that 
\begin{equation}\label{015}
H_{k}(\L(W_{\leq n-1}))=0\,\,\,\,\,\,\,\,\,\,\,\,\,,\,\,\,\,\,\,\,\,\,\,\,k\geq m.
\end{equation}
it follows
\begin{equation}\label{017}
H_{n-1}(\L(W_{\leq n-1}))=0.
\end{equation}
Let us consider the  Whitehead exact sequence of $(\L(W_{\leq n-1}),\delta)$, namely
\begin{equation}\label{013}
\cdots \rightarrow H_{k}(\L(W_{\leq n-1}))\rightarrow W_{k}\overset{b_{k}}{\longrightarrow }%
H_{k-1}(\L (W_{\leq k-2}))\rightarrow H_{k-1}(\L(W_{\leq n-1}))
\overset{}{\rightarrow }\cdots
\end{equation}
which implies
$$\cdots \rightarrow W_{n}=0\overset{b_{n}}{\longrightarrow }%
H_{n-1}(\L (W_{\leq n-2}))\rightarrow H_{n-1}(\L(W_{\leq n-1}))=0\to\cdots$$
As a result,  we obtain
\begin{equation}\label{103}
H_{n-1}(\L (W_{\leq n-2}))=0
\end{equation}
Now, according to (\ref{015}), the map   $b_{n-1}$ is an isomorphism and according to Corollary  \ref{c1},  it follows that
$$\mathcal{R}^{n-1}_{n-2}\cong \E(\L (W_{\leq n-2})).$$ 
Using  Theorem   \ref{t8}  (for $q=n-1$ and $k=n-2$) and (\ref{103}), we conclude that $\mathcal{E}(\L( W_{\leq n-1}))\overset{}{
	\cong}\mathcal{R}^{n-1}_{n-2}$. Consequently $\mathcal{E}(\L( W_{\leq n-1}))\cong \mathcal{E}(\L( W_{\leq n-2}))$

\noindent Using the same arguments  and taking in account the relation
(\ref{015}) which implies that  $b_{k}$  is an isomorphism for $k\geq m+1$,  by iterating the above process it follows that
$$\mathcal{E}(\L( W_{\leq n-1}))\cong \mathcal{E}(\L( W_{\leq n-2}))\cong\dots\cong\mathcal{E}(\L( W_{\leq m}))$$
Finally,  by the properties of the Sullivan and Quillen models and taking into  consideration Theorem \ref{t0} we obtain (\ref{016})
\end{proof}

\begin{theorem}\label{t06}
	Let $X$ be  a $\rm{2}$-connected rational  elliptic  space of formal dimension $n$.  If $\pi_{n}(X)\neq 0$, then the group  $\E(X)$ is infinite.
\end{theorem}
\begin{proof}
Since the formal dimension of $X$  is $n$, we can choose $(\L (W_{\leq
	n-1}),\delta)$ as the Quillen model of $X$ with $W_{n-1}\neq 0$ and $(\Lambda V^{\leq 2n-1},\partial)$ as its   Sullivan model.

\noindent Next, taking  $q=n-1 $ and $k=n-2$ in  (\ref{034}), we derive the following short exact sequence
	\begin{equation}\label{021}
	\mathrm{Hom}\big(W_{n-1}, H_{n-1}(\L (W_{\leq
		n-2}))\big) \rightarrowtail
	\mathcal{E}(\L( W_{\leq n-1}))\overset{}{
		\twoheadrightarrow}\mathcal{R}^{n-1}_{n-2}
	\end{equation}
Assume by contradiction that $\E(X)$ is finite. By Theorem \ref{t2}, we have $\E(\Lambda V^{\leq 2n-1})\cong \E(\Lambda V^{\leq n})$, so $\E(\Lambda V^{\leq n})$ is also finite and  the short exact sequence   (\ref{06})   implies that 
 $ H^{n}(\Lambda V^{\leq n-1})=0$ because $V^n\cong \pi_{n}(X)\neq 0$.  Taking in account that $X$ is 2-connected,  the formula (\ref{24}) implies that $H^{n}(\Lambda V^{\leq n-1})\cong H_{n-2}(\L (W_{\leq
		n-3 }))$. So $H_{n-2}(\L (W_{\leq
		n-3 }))=0$. Now recall that $\mathcal{R}^{n-1}_{n-2}$ is  
	the subgroup of $\aut(W_{n-1})\times \mathcal{E}(\L (W_{\leq
		n-3 }))$ consisting of the pairs $(\xi,[\alpha])$  making  the following
	diagram commutes
	
	\begin{picture}(300,90)(0,30)
	\put(60,100){$W_{n-1}\hspace{1mm}\vector(1,0){140}\hspace{1mm}W_{n-1}$}
	\put(69,76){\scriptsize $b_{n-1}$} \put(238,76){\scriptsize $b_{n-1}$}
	\put(66,96){$\vector(0,-1){38}$} \put(235,96){$\vector(0,-1){38}$}
	\put(155,103){\scriptsize $\xi$} \put(145,52){\scriptsize
		$H_{n-2}(\alpha)$} \put(35,46){$H_{n-2}(\L (W_{\leq n-3}))
		\hspace{1mm}\vector(1,0){110}\hspace{1mm}H_{n-2}(\L (W_{\leq n-3}))
		\hspace{1mm}$}
	\end{picture}

	\noindent As $H_{n-2}(\L (W_{\leq
		n-3 }))=0$, we deduce that $\mathcal{R}^{n-1}_{n-2}=\aut(W_{n-1})\times \mathcal{E}(\L (W_{\leq
		n-3 }))$  implying that  $\mathcal{R}^{n-1}_{n-2}$ is infinite and by (\ref{021}) the group 	$\mathcal{E}(\L( W_{\leq n-1}))$ is also infinite contradicting the fact that $\E(X)\cong 	\mathcal{E}(\L( W_{\leq n-1}))$ is finite.
\end{proof}

\subsection{ Self-closeness number $N\E(X)$} 
For a  space $X$, let  $[X,X]$ denote the monoid of homotopy classes of  self-maps of the  space $X$ and let
$$\A_{\#}^{k}(X)=\Big\{f\in[X,X]\,\mid\,\pi_{i}(f):\pi_{i}(X)\overset{\cong}{\longrightarrow} \pi_{i}(X) \text{ for any } i\leq k\}$$ 
In \cite{CL}, Choi and Lee  introduced the following concept:
\begin{definition}
	The  self-closeness number, denoted by $N\E(X)$, is defined as follows
	$$N\E(X)=min\Big\{k\,\mid\,\A_{\sharp}^{k}(X)=\E(X)\Big\}$$
\end{definition}
\begin{theorem}\label{t}
	If $X$ is a simply connected rational  elliptic  space of formal dimension $n$, then 
	$$N\E(X)\leq n$$
\end{theorem}
\begin{proof}
	As  $X$  is of  formal dimension $n$, by remark \ref{r1} its Sullivan  model has  the form  $(\Lambda V,\partial)=(\Lambda V^{\leq 2n-1},\partial)$. Define
	$$\A_{\#}^{k}(\Lambda V)=\Big\{[\alpha]\in[\Lambda V,\Lambda V]\,\mid\,\widetilde{\alpha}^{i}:V^{i}\overset{\cong}{\longrightarrow} V^{i} \text{ for any }  i\leq k\Big\},$$
	where 	$[\Lambda V,\Lambda V]$ denotes the monoid of homotopy cochain algebras of  self-equivalences classes  of $\Lambda V$ and where  $\widetilde{\alpha}$  is the graded linear isomorphism induced by $\alpha$ on the graded vector space of  indecomposables $V$ (see (\cite{FHT} 12,15.(d)).
	
	\noindent Clearly by virtue on the properties of the Sullivan model we can identify the two  sets  $\A_{\#}^{k}(\Lambda V)$ and $\A_{\#}^{k}(X)$.
	
	\noindent  First we have  $\A_{\#}^{2n-2}(\Lambda V)=\E(\Lambda V)$ and $\A_{\#}^{n}(\Lambda V^{\leq n})=\E(\Lambda V^{\leq n})$.  Next  it is easy to see that the set $\A_{\#}^{n}(\Lambda V^{\leq n})$ can be identified as a subset of   $\A_{\#}^{2n-1}(\Lambda V)$ by considering 
	the following  injective map
	\begin{equation}\label{1}
	\theta:\A_{\#}^{n}(\Lambda V^{\leq n}) \hookrightarrow\A_{\#}^{2n-2}(\Lambda V)\,\,\,,\,\,\,[\alpha]\longmapsto \theta([\alpha])=[\beta]
	\end{equation}
	where $\beta=\alpha$ on $V^{\leq n}$ and $\beta=id$ on $V^{> n}$.  
	Finally using Theorem \ref{t0} and (\ref{1})  we get 
	$$\A_{\#}^{n}(\Lambda V^{\leq n})\subseteq\A_{\#}^{2n-2}(\Lambda V)=\E(\Lambda V)=\E(\Lambda V^{\leq n})=\A_{\#}^{n}(\Lambda V^{\leq n})$$
	Therefore 
	$\A_{\#}^{n}(\Lambda V^{\leq n})=\E(\Lambda V)$ implying that $N\E(\Lambda V)\leq n$.
\end{proof}
\begin{example}\label{e2}
Given a finite group $G$. According to a Theorem of Frucht \cite{F},  there exists  a  connected finite graph $\G=(V,E)$,  where $V$ denotes the set of the  vertices of $\G$ and $E$  the set of its edges, such that $\aut(\G)\cong G$. Recall that in \cite{CV} Costoya and Viruel  constructed a free  commutative differential graded algebra  $$\big(\Lambda(x_{1},x_{2},y_{1},y_{2},y_{3},z,\{z_{v},x_{v}\}_{v\in V}),\partial\big)$$ 
 where the degrees of the elements  are  $$|x_{1}|=8\,\,,\,\,\,\,\,\,\,\,\,\,\,\,\,\,\,|x_{2}|=10\,\,\,\,\,\,\,\,,\,\,\,\,\,\,\,\,\, |x_{v}|=40\,\,\,\,\,\,\,\,,\,\,\,\,\,\,\,\,\, |z|=|z_{v}|=119$$ 
$$|y_{1}|=33\,\,\,\,\,\,\,\,,\,\,\,\,\,\,\,\,\,|y_{2}|=35\,\,\,\,\,\,\,\,,\,\,\,\,\,\,\,\,\,|y_{3}|=37$$
and where the differential is given by
\begin{eqnarray}
\partial(x_{1})&\hspace{-2mm}=&\hspace{-2mm}\partial(x_{2})=\partial(x_{v})=0\,,\,\,\,\,\,\partial(y_{1})=x^{3}_{1}x_{2},\,\,\,\,\,\partial(y_{2})=x^{2}_{1}x^{2}_{2},\,\,\,\,\,\partial(y_{3})=x_{1}x^{3}_{2}\nonumber\\
\partial (z_{v})&\hspace{-2mm}=&x_{v}^{3}+\overset{}{{\underset{(v,w)\in E}{\sum}}}x_{v}x_{w}x^{4}_{2}\,\,\,,\,\,\,\nonumber\\
\partial (z)&\hspace{-2mm}=&y_{1}y_{2}x^{4}_{1}x^{2}_{2}-y_{1}y_{3}x^{5}_{1}x_{2}+y_{2}y_{3}x^{6}_{1}+x^{15}_{1}+x^{12}_{2}\nonumber
\end{eqnarray}
and  proved that $\E(\Lambda(x_{1},x_{2},y_{1},y_{2},y_{3},\{z_{v},w_{v}\}_{v\in V}))\cong G$. 
Moreover they showed that $\big(\Lambda(x_{1},x_{2},y_{1},y_{2},y_{3},\{z_{v},w_{v}\}_{v\in V}),\partial\big)$  is an elliptic  of formal
dimension $n = 208 + 80|V|$ with $|V|$ the order of the graph $\G$.

Let $X$ be a rational space whose admits $\big(\Lambda(x_{1},x_{2},y_{1},y_{2},y_{3},z,\{z_{v},x_{v}\}_{v\in V}),\partial\big)$ as Sullivan model. On one hand and as  $X$ has formal dimension $n = 208 + 80|V|$  the Quillen model of $X$ can be written as  $(\L (W_{\leq 207 + 80|V|}),\delta)$ and because $X=X^{[120]}$ we deduce that  $N\E(X)\leq  120$. 

\noindent On the other hand we have $m=max\big\{k \,\,\,| \,\,\,\pi_{k}(X)\neq 0\big\}=max\big\{k \,\,\,| \,\,\,V^{k}\neq 0\big\}=119.$
 Therefore  applying Theorem \ref{t09} leads to
$$\E(X)\cong \E(X^{207 + 80|V|})\cong\dots\cong\E(X^{120})\cong G$$
Therefore we can ameliorate Costoya and Viruel  Theorem by reducing the formal dimension of  $X$  showing  that every finite group $G$ occurs as $G= \E(X)$ for some   rational space  $X^{120}$ having  formal dimension $n=120$. 
\end{example}
\begin{remark}
It is  important to notice 	that the space  $Z=X^{120}$ is not elliptic. Indeed, if $Z$ were elliptic, $H^{*}(Z,\Bbb Q)$ would be Poincar\'{e} duality (\cite{FHT}, Theorem A), and thus $\dim\,(H^{120}(Z,\Bbb Q))=1$. But an easy computation shows that  $\dim\,(H^{120}(Z,\Bbb Q))\geq |V|>1$.
	
\end{remark}
\begin{example}\label{e3}
In (\cite{AL1}, Example 5.2),  Arkowitz and Lupton constructed a free   commutative differential graded algebra   $\big(\Lambda(x_{1},x_{2},y_{1},y_{2},y_{3},z),\partial\big)$ with  $|x_{1}|=10$, $|x_{2}|=12$, $|y_{1}|=41$, $|y_{2}|=43$, $|y_{3}|=45$ and  $|z|=119$. The differential is given by
	\begin{eqnarray}
	\partial(x_{1})&\hspace{-2mm}=&\hspace{-2mm}\partial(x_{2})=0\,,\,\,\,\,\,\partial(y_{1})=x^{3}_{1}x_{2},\,\,\,\,\,\partial(y_{2})=x^{2}_{1}x^{2}_{2},\,\,\,\,\,\partial(y_{3})=x_{1}x^{3}_{2}\nonumber\\
	\partial (z)&\hspace{-2mm}=&y_{1}y_{2}x^{3}_{2}-y_{1}y_{3}x_{1}x_{2}^2+y_{2}y_{3}x^{2}_{1}x^{}_{2}+x^{12}_{1}+x^{10}_{2}\nonumber
	\end{eqnarray}
	and  proved that  $\E(\Lambda(x_{1},x_{2},y_{1},y_{2},y_{3},z)\cong\Bbb Z_{2}$.

\noindent 	Moreover they showed that $\big(\Lambda(x_{1},x_{2},y_{1},y_{2},y_{3},z),\partial\big)$  is an elliptic  of formal
	dimension $188$.
	
If $X$ is a rational space whose admits $\big(\Lambda(x_{1},x_{2},y_{1},y_{2},y_{3},z),\partial\big)$ as Sullivan model, then  its Quillen model can be written as  $(\L (W_{\leq 187}),\delta)$.  Since  $m=max\big\{k \,| \,\pi_{k}(X)\neq 0\big\}=max\big\{k \,| \,V^{k}\neq 0\big\}=119.$
	Therefore  applying Theorem \ref{t09} leads to
	$$\E(X)\cong \E(X^{187})\cong\dots\cong\E(X^{120})\cong \Bbb Z_{2}$$
\end{example}
\begin{example}\label{e4}
Define  $(\Lambda(x,y),\partial)$ with  $|x|=2p$ . The differential is given by
	\begin{equation}
	\partial(x)=0\,\,\,\,\,\,\,\,\,\,\,\,\,\,\,\,\,\,,\,\,\,\,\,\,\,\,\,\,\,\,\,\,\,\partial(y_{})=x^{a}\,\,\,\,\,\,\,\,\,\,\,\,\,\,\,\,\,\,,\,\,\,\,\,\,\,\,\,\,\,\,\,\,\,a\geq 2
	\end{equation}
Obviously 	$(\Lambda(x,y),\partial)$  is elliptic and it easy to see that its  formal
	dimension is $n=2(a-1)p$ and $m=max\big\{k \,| \,V^{k}\neq 0\big\}=2ap-1,$  so $m>n$. Now if $X$ is a rational space having  $(\Lambda(x,y),\partial)$ as the Sullivan model, then   by   applying Theorem \ref{t0} we get 
	$\E(X)\cong \E(\Lambda(x))\cong \Bbb Q-\{0\}.$
\end{example}

\vspace{.5cm}

\bibliographystyle{amsplain}

\end{document}